\documentclass[12pt]{amsart}
\usepackage{amssymb,amsmath,amsthm,hyperref,mathrsfs,multirow,xcolor}
\newtheorem{theorem}{Theorem}[section]
\newtheorem{lemma}[theorem]{Lemma}

\newtheorem*{theoremNoNum}{Theorem}

\theoremstyle{definition}

\theoremstyle{remark}
\newtheorem*{remark}{Remark}

\numberwithin{equation}{section}

\allowdisplaybreaks

\begin{document}
\newcommand\mylabel[1]{\label{#1}}
\newcommand{\beqs}{\begin{equation*}}
\newcommand{\eeqs}{\end{equation*}}
\newcommand{\beq}{\begin{equation}}
\newcommand{\eeq}{\end{equation}}
\newcommand\eqn[1]{(\ref{eq:#1})}
\newcommand\thm[1]{\ref{thm:#1}}
\newcommand\lem[1]{\ref{lem:#1}}
\newcommand\propo[1]{\ref{propo:#1}}
\newcommand\corol[1]{\ref{cor:#1}}
\newcommand\sect[1]{\ref{sec:#1}}
\newcommand\subsect[1]{\ref{subsec:#1}}

\newcommand{\Z}{\mathbb Z}
\newcommand{\N}{\mathbb N}
\newcommand{\R}{\mathbb R}
\newcommand{\C}{\mathbb C}
\newcommand{\Q}{\mathbb Q}
\newcommand{\leg}[2]{\genfrac{(}{)}{}{}{#1}{#2}}
\newcommand{\bfrac}[2]{\genfrac{}{}{}{0}{#1}{#2}}
\newcommand{\sm}[4]{\left(\begin{smallmatrix}#1&#2\\ #3&#4 \end{smallmatrix} \right)}

\title[Bilateral series and Ramanujan's radial limits]
{Bilateral series and Ramanujan's radial limits}

\author{J. Bajpai}
\address{Department of Mathematical and Statistical Sciences, University of Alberta, Edmonton, Alberta T6G 2G1. Canada.}
\email{jitendra@math.ualberta.ca}

\author{S. Kimport}
\address{Department of Mathematics, Yale University, New Haven, CT. 06520}
\email{susie.kimport@yale.edu}

\author{J. Liang}
\address{Department of Mathematics, University of Florida, Gainsville, FL. 32601}
\email{jieliang@ufl.edu}
\thanks{This project is the result of participation in the 2013 Arizona Winter School.}

\author{D. Ma}
\address{Department of Mathematics, University of Arizona, Tucson, AZ. 85721}
\email{martin@math.arizona.edu}

\author{J. Ricci}
\address{Department of Mathematics, Wesleyan University, Middletown, CT. 06459}
\email{jricci@wesleyan.edu}

\date{May 12, 2013}

\begin{abstract}
     Ramanujan's last letter to Hardy explored the asymptotic properties of modular forms, as well as those of certain interesting $q$-series which he called \emph{mock theta functions}. For his mock theta function $f(q)$, he claimed that as $q$ approaches an even order $2k$ root of unity $\zeta$,
     \[\lim_{q\to \zeta} \big(f(q) - (-1)^k (1-q)(1-q^3)(1-q^5)\cdots (1-2q + 2q^4 - \cdots)\big) = O(1),\]
     and hinted at the existence of similar statements for his other mock theta functions.  Recent work of Folsom-Ono-Rhoades provides a closed formula for the implied constant in this radial limit of $f(q)$. Here, by different methods, we prove similar results for all of Ramanujan's 5th order mock theta functions. Namely, we show that each 5th order mock theta function may be related to a modular bilateral series, and exploit this connection to obtain our results.  We further explore other mock theta functions to which this method can be applied.

\end{abstract}

\maketitle

\section{Introduction}

    In his deathbed letter to Hardy in 1920, Ramanujan wrote down 17 curious $q$-series which he dubbed \emph{mock theta functions}. Due to work of Zwegers \cite{Zwegers01, Zwegers02}, Bringmann-Ono and others \cite{Zagier}, we are now able to recognize Ramanujan's mock theta functions as holomorphic parts of weight $1/2$ harmonic weak Maass forms. Although this has been a catalyst for recent developments in numerous areas of mathematics, here we will focus on Ramanujan's original formulation.

    Ramanujan analyzed $q$-hypergeometric series with asymptotics similar to those of modular theta functions near roots of unity, but which were not themselves modular.   In his letter, he asked of such series:
    
    \noindent\emph{``...The question is: - is the function taken the sum of two functions one of which is an ordinary theta function and the other a (trivial) function which is O(1) at \underline{all} the points $e^{2\pi im/n}$? ...I have constructed a number of examples in which it is inconceivable to construct a $\vartheta$-function to cut out the singularities of the original function.''}
    
    Examples of this form are what he then referred to as mock theta functions.  Though Ramanujan did not  prove his assertion,  recent work by Griffin-Ono-Rolen \cite{GOR} has confirmed that no such theta functions exist for Ramanujan's examples.  The only example Ramanujan offered details for is the function
    \[f(q) := 1+\frac{q}{(1+q)^2}+\frac{q^4}{(1+q)^2(1+q^2)^2}+\cdots\]
    He claimed that as $q$ approaches an even order $2k$ root of unity $\zeta$ radially within the unit disk, we have that
    \begin{align}\label{fbo}
        \lim_{q \to \zeta} \big( f(q) - (-1)^k b(q) \big) = O(1),
    \end{align}
   where $b(q) := (1-q)(1-q^3)(1-q^5) \cdots (1-2q+2q^4-\cdots)$ is a modular form.
   
   \begin{remark}    Here and throughout this paper, we take $q := e^{2\pi i \tau}$ for $\tau\in\mathbb{H}$.  Then, by modular form we will mean the function is modular, up to a rational power of $q$, with respect to some character.  The character for these modular forms can be explicitly calculated using \cite{Ono}, for example.
\end{remark}

    Recently, Folsom, Ono, and Rhoades \cite{FOR1, FOR2} provided a closed formula for the $O(1)$ numbers in Ramanujan's claim \eqref{fbo}:
\begin{theoremNoNum}[Theorem 1.1 in \cite{FOR1,FOR2}] If $\zeta$ is a primitive even order $2k$ root of unity, then, as $q$ approaches $\zeta$ radially within the unit disk, we have that
    \begin{equation*}
      \lim_{q \to \zeta} \big( f(q) - (-1)^k b(q) \big) = -4\sum_{n=0}^{k-1} (1+\zeta)^2(1+\zeta^2)^2\cdots (1+\zeta^n)^2\zeta^{n+1}.
    \end{equation*}
\end{theoremNoNum}
    To obtain their closed form, they utilized a bilateral series associated to $f(q)$ that is not a modular form, but does have similar asymptotics to $b(q)$.  The goal of this paper is to show that a bilateral series naturally associated to a given mock theta function can sometimes be used to not only provide a similar closed formula, but can also be modular and play the role of $b(q)$.  When this is the case, we obtain the closed formulas by using a different method of proof than that employed in \cite{FOR1, FOR2}.

    We first consider Ramanujan's 5th order mock theta functions:

    \begin{equation*}\label{5thorders}
    \begin{split}
        f_0(q) &:= \sum_{n \geq 0} \frac{q^{n^2}}{(-q; q)_n} \\
        \psi_0(q) &:= \sum_{n \geq 0} q^{(n+1)(n+2)/2} (-q; q)_n \\
        \phi_0(q) &:= \sum_{n \geq 0} q^{n^2} (-q; q^2)_n  \\
        F_0(q) &:= \sum_{n \geq 0} \frac{q^{2n^2}}{(q; q^2)_n}\\
        \chi_0(q) &:= \sum_{n \geq 0} \frac{q^n}{(q^{n+1};q)_n}
    \end{split}
    \hspace{15mm}
    \begin{split}
        f_1(q) &:= \sum_{n \geq 0} \frac{q^{n(n+1)}}{(-q; q)_n} \\
        \psi_1(q) &:= \sum_{n \geq 0} q^{n(n+1)/2} (-q; q)_n \\
        \phi_1(q) &:= \sum_{n \geq 0} q^{(n+1)^2} (-q; q^2)_n  \\
        F_1(q) &:= \sum_{n \geq 0} \frac{q^{2n(n+1)}}{(q; q^2)_{n+1}}\\
        \chi_1(q) &:= \sum_{n \geq 0} \frac{q^n}{(q^{n+1};q)_{n+1}}
    \end{split}
    \end{equation*}
	where $(a;q)_n$ is the $q$-Pochhammer symbol defined as
    \[ (a; q)_n = \frac{(a; q)_\infty}{(aq^n; q)_\infty}, \]
    where $(a; q)_\infty := (1-a)(1-aq)(1-aq^2)\cdots$ and $n\in \Z$.  From this formulation, we see the well-known form (see \cite{Fine}, \cite{Gasper} for example) 
    \begin{align}
        (a; q)_{-n} &= \frac{(-a)^{-n}q^{n(n+1)/2}}{(a^{-1}q; q)_n}. \label{negative}
    \end{align}

    For a mock theta function $M(q):=\sum_{n \geq 0} c(n; q)$, we define its associated bilateral series by $B(M; q) := \sum_{n \in \Z} c(n; q)$. For example, the bilateral series $B(f_0;q)$ is given by
    \[ B(f_0;q) := \sum_{n \in \Z} \frac{q^{n^2}}{(-q; q)_n}. \]

      Surprisingly, the bilateral series associated to the 5th order mock theta functions are in fact modular forms (with the exception of $B(\chi_0;q)$ and $B(\chi_1;q)$ which will be addressed in Section \ref{chi0chi1}). Since a key component of Ramanujan's claim \eqref{fbo} is the fact that $b(q)$ is a modular form, these bilateral series beautifully lend themselves to similar radial limits.  Moreover, these bilateral series can be written as linear combinations of mock theta functions which then reveal the following simple closed formulas similar to Theorem 1.1 in \cite{FOR1}.

    \begin{theorem}\label{thm:cleanproof}
        Let $\zeta$ be a primitive root of unity, $k\in\N$, and suppose $q\to\zeta$ radially within the unit disk:
        \begin{itemize}
        \item[(a)] If $\zeta$ has order $2k$, then we have that
            \begin{align*}
        	   \lim_{q\to\zeta} \big( f_0(q) - B(f_0; q) \big)
                 &= -2\sum_{n=0}^{k-1} \zeta^{(n+1)(n+2)/2}(-\zeta;\zeta)_n,\\
        	   \lim_{q\to\zeta} \big( f_1(q) - B(f_1; q) \big)
                 &= -2\sum_{n=0}^{k-1} \zeta^{n(n+1)/2}(-\zeta;\zeta)_n,
            \end{align*}
            where $B(f_0;q)$ and $B(f_1;q)$ are modular forms of weight $1/2$ with level $\Gamma_1(20)$.

        \item[(b)] If $\zeta$ has order $2k-1$, then we have that
            \begin{align*}
        	   \lim_{q\to\zeta} \big( F_0(q) - B(F_0; q)\big)
                 &= 1 - \sum_{n=0}^{k-1} (-\zeta)^{n^2}(\zeta;\zeta^2)_n, \\
        	   \lim_{q\to\zeta} \big( F_1(q) - B(F_1; q)\big)
                 &= \zeta^{-1}\sum_{n=0}^{k-1} (-\zeta)^{(n+1)^2}(\zeta;\zeta^2)_n,
            \end{align*}
            where $B(F_0;q)$ and $B(F_1;q)$ are modular forms of weight $1/2$ with level $\Gamma_1(10)$.
        \end{itemize}
    \end{theorem}

    \begin{theorem}\label{thm:messyproof}
       Let $\zeta$ be a primitive root of unity, $k\in\N$, and suppose $q \to \zeta$ radially within the unit disk:
        \begin{itemize}
          \item[(a)] If $\zeta$ has order $2k-1$, then we have that
             \begin{align*}
        	   \lim_{q\to\zeta} \big( \psi_0(q) - B(\psi_0;q)\big)
                 &= -\frac{1}{2} -\sum_{n=1}^{2k-1} \frac{\zeta^{n^2}}{(-\zeta;\zeta)_n}, \\
        	   \lim_{q\to\zeta} \big( \psi_1(q) - B(\psi_1;q)\big)
                 &= -\frac{1}{2} -\sum_{n=1}^{2k-1} \frac{\zeta^{n(n+1)}}{(-\zeta;\zeta)_n},
            \end{align*}
            where $B(\psi_0;q)$ and $B(\psi_1;q)$ are modular forms of weight $1/2$ with level $\Gamma_1(20)$.

        \item[(b)] If $\zeta$ has order $m$, then we have that
             \begin{align*}
               \lim_{q\to\zeta} \big( \phi_0(q) - B(\phi_0;q)\big)
                 &= \begin{cases}
                    \displaystyle -2\sum_{n=1}^{2k-1} \frac{\zeta^{2n^2}}{(-\zeta;\zeta^2)_n} & \text{if } m = 2k-1, \\
                    \displaystyle -2\sum_{n=1}^{2k} \frac{\zeta^{2n^2}}{(-\zeta;\zeta^2)_n} &  \text{if } m = 4k,
                 \end{cases} \\
        	   \lim_{q\to\zeta} \big( \phi_1(q) - B(\phi_1;q)\big)
                 &= \begin{cases}
                    \displaystyle -2\zeta\sum_{n=1}^{2k-1} \frac{\zeta^{2n(n-1)}}{(-\zeta;\zeta^2)_n} &  \text{if } m = 2k-1, \\
                    \displaystyle -2\zeta\sum_{n=1}^{2k} \frac{\zeta^{2n(n-1)}}{(-\zeta;\zeta^2)_n} &  \text{if } m = 4k,
                 \end{cases}
            \end{align*}
            where $B(\phi_0;q)$ and $B(\phi_1;q)$ are modular forms of weight $1/2$ with level $\Gamma_1(10)$.
        \end{itemize}
    \end{theorem}

    \noindent \emph{Four remarks.}
        \begin{itemize}
          \item[(a)] These theorems cover all of the roots of unity where the mock theta functions $f_0$, $f_1$, $F_0$, $F_1$, $\psi_0$, $\psi_1$, $\phi_0$, and $\phi_1$ have singularities.
          \item[(b)] The hypotheses on the roots of unity in Theorem \ref{thm:messyproof} ensure that the denominators in the closed formulas do not cause singularities.
          \item[(c)] The modular forms $B(M; q)$ are given explicitly in Section \ref{modularityProof}.
          \item[(d)] These results are particularly elegant because only one modular form is needed to cut out all of the singularities.  This differs from Ramanujan's claim \eqref{fbo} where the modular form changes by a factor of $(-1)^k$ depending on which even order $2k$ root of unity is being considered.
        \end{itemize}

    These theorems, coupled with the fact that the bilateral series are modular forms, provide insight into why Ramanujan might have been so fascinated with these functions.  The above formulas embody the  property that, although mock theta functions are not themselves modular, they do have similar asymptotic properties to modular forms.   Additionally, the closed formulas for each of the radial limits suggest further connections between mock theta functions and quantum modular forms (see \cite{FOR1}).

    In Section \ref{modularity} we prove the modularity of the bilateral series in Theorems \ref{thm:cleanproof} and \ref{thm:messyproof}. In Section \ref{proofs} we provide the proofs of these theorems.  In order to obtain results for $\chi_0$ and $\chi_1$, we use a different type of bilateral series, which is covered in Section \ref{chi0chi1}.  Section \ref{otherorders} develops similar results for other mock theta functions where this method using bilateral series can be applied.

\section*{Acknowledgement}

    The authors would like to thank Amanda Folsom, Ken Ono, and Robert Rhoades for their advice and guidance. The authors would also like to thank the organizers of the 2013 Arizona Winter School for providing the environment in which this project was developed.

\section{Modularity of the bilateral series}\label{modularity}
    As mentioned in the introduction, a key component in Theorems \ref{thm:cleanproof} and \ref{thm:messyproof} is that the bilateral series are in fact modular forms.  This section is devoted to showing the modularity of these series, as is summarized in the following lemma.

    \begin{lemma}\label{5thOrderModularity}
    The bilateral series $B(M;q)$ is a modular form of weight $1/2$ with level $\Gamma_1(20)$ when $M \in \{f_0, f_1, \psi_0, \psi_1\}$ and level $\Gamma_1(10)$ when $M\in\{\phi_0, \phi_1, F_0, F_1\}$.
    \end{lemma}

    To prove this, we first rewrite our bilateral series as linear combinations of mock theta functions.  This then allows us to utilize the mock theta conjectures, which relate these linear combinations to specific modular forms.  Using the work in Section \ref{rogersRamanujan}, we are then able to determine weight and level.

    \subsection{Alternative forms of bilateral series}\label{bsums}
    By using \eqref{negative}, we can express the bilateral series associated to each of the 5th order mock theta functions in Lemma \ref{5thOrderModularity} as follows:

        \begin{align}
            B(f_0; q) &= f_0(q)+2\psi_0(q) = 2B(\psi_0;q)\label{bf0psi0}, \\
            B(f_1; q) &= f_1(q)+2\psi_1(q) = 2B(\psi_1;q)\notag, \\
            B(F_0; q) &= F_0(q) + \phi_0(-q) -1 =  B(\phi_0; -q), \label{bF0phi0}\\
            B(F_1; q) &= F_1(q) - q^{-1}\phi_1(-q)  =  -q^{-1}B(\phi_1; -q) \notag.
        \end{align}

    \subsection{Reformulations of the Rogers-Ramanujan functions}\label{rogersRamanujan}
        In our proof of Lemma \ref{5thOrderModularity} we will make use of the Rogers-Ramanujan functions,

        \begin{align*}
            G(q) := \sum_{n \geq 0} \frac{q^{n^2}}{(q; q)_n} = \frac{1}{(q; q^5)_\infty(q^4; q^5)_\infty} \text{\quad and \quad} H(q) := \sum_{n \geq 0} \frac{q^{n^2+n}}{(q; q)_n} = \frac{1}{(q^2; q^5)_\infty(q^3; q^5)_\infty}.
        \end{align*}

        We provide a reformulation of $G(q)$ and $H(q)$ in terms of the Dedekind $\eta$ function  $\eta(\tau) := q^{1/24}(q;q)_\infty$ and Klein forms as defined below. This will be helpful in determining the weight and level of the bilateral series described in the lemma. Though these functions have been studied in great detail, we were unable to find these reformulations in the literature and therefore we include them for completeness.

        The Klein form $\mathfrak{t}^{(N)}_{(r,s)}=\mathfrak{t}_{(r,s)}$ for $N\in \mathbb{N}$ and $(r,s)\in\Z^2$ such that $(r,s)\not\equiv (0,0)\!\! \mod{N\times N}$ is a function on $\mathbb{H}$ defined as
 \begin{align}\mathfrak{t}_{(r,s)}(\tau) := -\frac{\zeta_{2N^2}^{s(r-N)}}{2\pi i}q^{\frac{r(r-N)}{2N^2}} (1- \zeta_N^s q^{\frac{r}{N}}) \prod_{n=1}^\infty \frac{(1-\zeta_N^s q^{n+\frac{r}{N}})(1-\zeta_N^{-s} q^{n-\frac{r}{N}})}{(1-q^n)^2},\label{klein}\end{align}
 where $\zeta_n := e^{2\pi i /n}$.
       
        The Klein form has the following transformation law under $\gamma=\sm{a}{b}{c}{d}\in\textnormal{SL}_2(\Z)$:
        \[\mathfrak{t}_{(r,s)} (\gamma \tau) = (c\tau + d)^{-1}\mathfrak{t}_{(r,s)\gamma}(\tau)\]
        where $(r,s)\gamma = (ra + sc, rb + sd)$. For more details on Klein forms see \cite{Kubert}.

        \begin{lemma}\label{RRquotients}
            Assuming the notation above,
            \[G(q) =  -\frac{\zeta_{5}^{3}}{2\pi i} \frac{q^{\frac{1}{60}}}{\eta^{2}(5\tau) \mathfrak{t}_{(1,5)}(5\tau)} \text{\quad and \quad} H(q) =-\frac{\zeta_{10}^{7}}{2\pi i} \frac{q^{-\frac{11}{60}}}{\eta^{2}(5\tau) \mathfrak{t}_{(2,5)}(5\tau)}.\]

        \end{lemma}

\begin{proof}
By taking $r=1, s=5$ and $N=5$ in (\ref{klein}), we get that 
\begin{align*}
\mathfrak{t}_{(1,5)}(\tau) &= -\frac{\zeta_{5}^{-2}}{2\pi i}q^{-\frac{2}{25}} (1- q^{\frac{1}{5}}) \prod_{n=1}^\infty \frac{(1-q^{n+\frac{1}{5}})(1- q^{n-\frac{1}{5}})}{(1-q^n)^2}.
\end{align*}
Then, letting $\tau\mapsto 5\tau$, we see
\[\mathfrak{t}_{(1,5)}(5\tau) = -\frac{\zeta_{5}^{3}q^{\frac{1}{60}}}{2\pi i} \frac{(q;q^5)_{\infty} (q^4;q^5)_{\infty}}{\eta^{2}(5\tau)}.\]

This implies that
\begin{align}
	-\zeta_{5}^{-3}{(2\pi i)}q^{-\frac{1}{60}}\mathfrak{t}_{(1,5)}(5\tau) \eta^{2}(5\tau) = (q;q^5)_{\infty} (q^4;q^5)_{\infty}.\label{kl}
\end{align} 
Now, by using the definition of $G(q)$ and equation (\ref{kl}), we immediately obtain the desired form for $G(q)$. Similarly, by taking $r=2$, $s=5$, $N=5$ and letting $\tau\mapsto 5\tau$, we get the result for $H(q)$.
\end{proof}
Note that $\mathfrak{t}_{(r,N)}^{(N)}(N\tau)$ is a modular form of weight $-1$ and level $\Gamma_1(N)$.  Therefore from Lemma \ref{RRquotients}, we immediately see that the Rogers-Ramanujan functions $G$ and $H$ are modular forms of weight $0$ and level $\Gamma_{1}(5)$.

\subsection{Proof of Lemma \ref{5thOrderModularity}}\label{modularityProof}
    We start by proving the modularity of $B(f_0;q)$.  For this, we use four of the identities established by Watson \cite{Watson} (reprinted as $(2.13)_R-(2.16)_R$ in \cite{AG}),
    \begin{equation} \label{Ci}
        \begin{split}
            C_1(q) &:= f_0(q) + 2F_0(q^2) - 2 - \vartheta_4(0;-q)G(-q) = 0,  \\
    	    C_2(q) &:= \phi_0(-q^2) + \psi_0(q) - \vartheta_4(0;-q)G(-q) = 0,  \\
    	    C_3(q) &:= 2\phi_0(-q^2) - f_0(q) - \vartheta_4(0;q)G(q) = 0, \\
    	    C_4(q) &:= \psi_0(q) - F_0(q^2) + 1 - qK(q^2)H(q^4) = 0,
        \end{split}    	
	\end{equation}
	where
	\begin{align*}
		\vartheta_4(0; q) := \sum_{n \in \Z} (-1)^nq^{n^2} = \frac{\eta(\tau)^2}{\eta(2\tau)}\text{\quad and\quad}
		K(q) := \sum_{n \geq 0} q^{n(n+1)/2} = \frac{\eta(2\tau)^2}{q^{1/8}\eta(\tau)}.
	\end{align*}
	Note that to obtain $C_1(q)$ and $C_2(q)$ in \eqref{Ci}, we let $q\mapsto -q$ in $(2.13)_R$ and $(2.14)_R$, respectively, from \cite{AG}.  In this form, it is not hard to see that
	\[4C_4(q) - 2C_2(q) + C_3(q) +2C_1(q) = f_0(q) + 2\psi_0(q) - \vartheta_4(0;q)G(q) + 4qK(q^2)H(q^4) = 0,\]
	that is,
	\[B(f_0;q) = f_0(q) + 2\psi_0(q) = \vartheta_4(0;q)G(q) + 4qK(q^2)H(q^4).\]
	We now make use of the well-known fact that, for $N\in\N$, $\eta(N\tau)$ is a weight $1/2$ modular form on $\Gamma_0(N) \supseteq \Gamma_1(N)$ (see, for example, \cite{Ono}).  Then, using Lemma \ref{RRquotients} and the transformation of the Klein forms, we see that $B(f_0;q)$ is a weight $1/2$ modular form on $\Gamma_1(20)$, as desired.  Since $B(f_0;q) = 2B(\psi_0;q)$, we have shown that $B(\psi_0;q)$ is a modular form as well.
	
	The modularity of the remaining bilateral series is obtained similarly: using the mock theta conjectures in Section 2 of \cite{AG}, also proved by Watson, we can show that
	\[B(f_1;q) = 2B(\psi_1;q) = -\vartheta_4(0;q)H(q) + 4K(q^2)G(q^4).\]
	Further, using the mock theta conjectures in Section 3 of \cite{AG}, proved by Hickerson in \cite{Hickerson}, we can show that
          \begin{align*}
            B(F_0; q) =  F_0(q) + \phi_0(-q)-1 &= \frac{(q^5; q^5)_\infty G(q^2)H(q)}{H(q^2)} - qK(q^5)H(q^2), \\
            B(F_1; q) = F_1(q) -q^{-1}\phi_1(-q) &= 3K(q^5)G(q^2) - \frac{H(q)^2 (q^5;q^5)_\infty}{G(q)}.
        \end{align*}
        Then, since $B(F_0;q) = B(\phi_0; -q)$ and $B(F_1;q) = -q^{-1}B(\phi_1;-q) $,  we have shown the modularity of each of these bilateral series.\hfill \ensuremath{\Box}

\section{Proof of theorems \ref{thm:cleanproof} and \ref{thm:messyproof}}\label{proofs}

\subsection{Proof of Theorem \ref{thm:cleanproof}}
    \begin{proof}
        We start by considering $f_0$.  Let $\zeta$ be a primitive root of unity of order $2k$. From \eqref{bf0psi0}, we have that
            \begin{align*}
                \lim_{q\to\zeta} \big(f_0(q) - B(f_0; q)\big) &= -2\lim_{q\to\zeta} \sum_{n\ge 0} q^{(n+1)(n+2)/2}(-q;q)_n\\
        		  & = -2\lim_{r\to 1^-} \sum_{n\ge 0} (r\zeta)^{(n+1)(n+2)/2}(-r\zeta;r\zeta)_n \\
                	&= -2\sum_{n=0}^{k-1} \zeta^{(n+1)(n+2)/2}(-\zeta;\zeta)_n.
            \end{align*}
            We can interchange the limit and the summation since, as shown in Section 6 of \cite{Watson},  the above infinite sum is absolutely convergent for all $0 < r \le 1$.  Then, the last equality follows from the fact that, since $\zeta$ has order $2k$,
            \[(-\zeta;\zeta)_n = (1+\zeta) (1+\zeta^2)\cdots (1+\zeta^{k-1}) (1+\zeta^k) \cdots (1+\zeta^n) = 0\]
        for all $n\ge k$. The modularity of $B(f_0;q)$ is shown in Lemma \ref{5thOrderModularity}.  This proves the result for $f_0$.
        The remaining statements of the theorem are proved in a similar fashion.
    \end{proof}

\subsection{Proof of Theorem \ref{thm:messyproof}}

    \begin{proof}
        We start by considering $\psi_0$. Let $\zeta$ be a primitive root of unity of order $2k-1$. From \eqref{bf0psi0}, we have that
        \begin{equation}
            \lim_{q\to\zeta} \big( \psi_0(q) - B(\psi_0;q) \big)= -\frac{1}{2}\lim_{q\to\zeta}\left( \sum_{n=0}^{\infty} \frac{q^{n^2}}{(-q;q)_n} \right) = -\frac{1}{2}\lim_{r\to1^{-}}\left( \sum_{n=0}^{\infty} \frac{(r\zeta)^{n^2}}{(-r\zeta;r\zeta)_n} \right).\label{radialpsi0}
        \end{equation}
        We will write this limit as a finite sum.  Due to the shape of the series, the proof in this case is more subtle.  Therefore, we provide more details.

        For any complex number $u=\cos\theta+i\sin\theta$ of modulus 1, let us consider $|\frac{1}{r}+u|$ as a function of $r$. For $0< r\leq 1$, the function $|\frac{1}{r}+u|$ has a minimum at $r=1$ since
        $$\frac{d}{dr}\left|\frac{1}{r}+u\right|=\frac{-(1+r\cos\theta)}{r^3\sqrt{\frac{1}{r^2}+\frac{2}{r}\cos\theta+1}}\leq 0.$$
        Using this, we can construct a dominating series, which shows that the rightmost sum in \eqref{radialpsi0} is absolutely convergent for all $0<r \leq 1$. Therefore, we can interchange the limit with the infinite sum, yielding\\
    	    \begin{align*}
                    \lim_{q\to\zeta} \big( \psi_0(q) - B(\psi_0;q) \big)
                        &= -\frac{1}{2}\lim_{r\to1^-} \sum_{n=0}^{\infty} \frac{(r\zeta)^{n^2}}{(-r\zeta;r\zeta)_n} \\
    		   		    &= -\frac{1}{2} \sum_{n=0}^{\infty} \frac{\zeta^{n^2}}{(-\zeta;\zeta)_n}\\
                        &= -\frac{1}{2}\left( 1 + \sum_{N=0}^\infty \sum_{n=1}^{2k-1}
                            \frac{\zeta^{(N(2k-1)+n)^2}}{(1+\zeta)\cdots(1+\zeta^{N(2k-1)+n})} \right)\\
                        &= -\frac{1}{2}\left( 1 + \sum_{N=0}^\infty \frac{1}{2^N}\sum_{n=1}^{2k-1}
                            \frac{\zeta^{n^2}}{(1+\zeta)\cdots(1+\zeta^{n})} \right) \\
                        &= -\frac{1}{2} -\sum_{n=1}^{2k-1} \frac{\zeta^{n^2}}{(-\zeta;\zeta)_n}.
                \end{align*}
       The modularity of $B(\psi_0;q)$ is shown in Lemma \ref{5thOrderModularity}. The remaining statements in the theorem are proved in a similar fashion.
    \end{proof}

\section{Results for $\chi_0$ and $\chi_1$}\label{chi0chi1}

    To complete our study of Ramanujan's 5th order mock theta functions, we now consider
    $\chi_0$ and $\chi_1$. As previously mentioned, the bilateral series $B(\chi_0;q)$ and $B(\chi_1;q)$ are not modular forms.  In fact, $B(\chi_0;q)$ and $B(\chi_1;q)$ are not even defined on any half-plane. However, we can modify our definition of the bilateral series for these two mock theta functions to attain similar results.

    To do so, we use the following alternative forms for $\chi_0$ and $\chi_1$ from the mock theta conjectures:
        \begin{align*}
            \chi_0(q) &= 2F_0(q)-\phi_0(-q),\\
            \chi_1(q) &= 2F_1(q) + q^{-1}\phi_1(-q).
        \end{align*}
        We then define the bilateral series associated to $\chi_0$ and $\chi_1$ as follows:
        \begin{align*}
            \mathcal{B}(\chi_0;q) &:=  2B(F_0;q)-B(\phi_0;-q) = B(F_0;q), \\
            \mathcal{B}(\chi_1;q) &:= 2B(F_1;q)+q^{-1}B(\phi_1;-q) = B(F_1;q),
        \end{align*}
        where the second equality in each definition follows from \eqref{bF0phi0}.
        We can then use Lemma \ref{5thOrderModularity} to immediately show that $\mathcal{B}(\chi_0;q)$ and $\mathcal{B}(\chi_1;q)$ are modular forms of weight $1/2$ with level $\Gamma_1(10)$.  Therefore, we can use this definition of the bilateral series for $\chi_0$ and $\chi_1$ to obtain analogous radial limits for these mock theta functions.

    \begin{theorem}\label{thm:chi}
        Let $\zeta$ be a primitive root of unity, $k\in\N$, and suppose $q\to\zeta$ radially within the unit disk:
        \begin{itemize}
        \item[(a)] If $\zeta$ has order $2k-1$, then we have that 
            \begin{align*}
        	\lim_{q\to\zeta} \big( \chi_0(q) - 2\mathcal{B}(\chi_0;q)\big)
                &= -3\sum_{n=0}^{k-1} (-\zeta)^{n^2}(\zeta;\zeta^2)_n+2, \\
	        \lim_{q\to\zeta} \big( \chi_1(q) - 2\mathcal{B}(\chi_1;q)\big)
                &= 3\zeta^{-1}\sum_{n=0}^{k-1} (-\zeta)^{(n+1)^2}(\zeta;\zeta^2)_n.
            \end{align*}

        \item[(b)] If $\zeta$ has order $2k$, then we have that
            \begin{align*}
        	\lim_{q\to\zeta} \big(\chi_0(q) + \mathcal{B}(\chi_0;q)\big)
                &= 6\sum_{n=1}^{k} \frac{\zeta^{2n^2}}{(\zeta;\zeta^2)_n}+2,\\
	        \lim_{q\to\zeta} \big( \chi_1(q) + \mathcal{B}(\chi_1;q)\big)
                &= 6\sum_{n=1}^{k} \frac{\zeta^{2n(n-1)}}{(\zeta;\zeta^2)_n}.
        \end{align*}
        \end{itemize}
        Moreover, $\mathcal{B}(\chi_0;q)$ and $\mathcal{B}(\chi_1;q)$ are modular forms of weight $1/2$ with level $\Gamma_1(10)$.
    \end{theorem}

Note that the denominators indicated in part (b) of Theorem \ref{thm:chi} do not vanish since $\zeta$ has order $2k$.

\begin{proof}
	To show part (a), we note that
	\[\chi_0(q) - 2\mathcal{B}(\chi_0;q) = -3\phi_0(-q) + 2 \text{\quad and \quad } \chi_1(q) - 2\mathcal{B}(\chi_1;q) = 3\phi_1(-q).\] 
	Then the result follows directly from Theorem \ref{thm:cleanproof}.  To show part (b), we note that 
	\[\chi_0(q) + \mathcal{B}(\chi_0;q) = 3F_0(-q) + 2 \text{\quad and \quad } \chi_1(q) + \mathcal{B}(\chi_1;q) = 3F_1(-q).\] 
	Then, using a similar argument to that in the proof of Theorem \ref{thm:messyproof}, we obtain the desired result.
\end{proof}

\section{Mock theta functions of other orders}\label{otherorders}

	Having fully explored all of Ramanujan's 5th order mock theta functions using the method of bilateral series, we now turn our attention to applying this method to mock theta functions of other orders.  We discuss in this section the cases where bilateral series can be exploited to give results akin to those listed above.
	
	We start by exploring 3rd order mock theta functions.  Then we move on to those of even order.

\subsection{3rd order mock theta functions}
    The method of looking at the associated bilateral series can be applied to the following three mock theta functions:
    \begin{equation*}\label{3rdorderMock}
        \phi(q) := \sum_{n=0}^\infty \frac{q^{n^2}}{(-q^2;q^2)_n} \qquad
        \psi(q):= \sum_{n=1}^\infty \frac{q^{n^2}}{(q;q^2)_n} \qquad
        \nu(q) := \sum_{n=0}^\infty \frac{q^{n(n+1)}}{(-q;q^2)_{n+1}}.
    \end{equation*}

    We form the bilateral series as before.  Using \eqref{negative} and equation (6.1) from \cite{Fine},
    \begin{equation}\label{Fine61}
        \sum_{n\ge 0} (aq; q)_n t^n = \frac{1}{1-t}\sum_{n\ge 0} \frac{(-at)^n q^{(n^2+n)/2}}{(tq; q)_n},
    \end{equation}
    it is not hard to show that
    \begin{align*}
    	B(\phi; q) = \phi(q) + 2\psi(q) = 2B(\psi; q) \hspace{6mm} \text{ and } \hspace{6mm}
    	B(\nu; q) = \nu(q) + \nu(-q).
    \end{align*}

    We first establish the modularity of these bilateral series.

    \begin{lemma}\label{3rdOrderModular}
        $B(\phi;q)$, $B(\psi;q)$, and $B(\nu;q)$ are modular forms of weight $1/2$ and level $\Gamma_0(4)$.
    \end{lemma}

    \begin{proof}
        We make use of Ramanujan's ${}_1\psi_1$ formula,
        $${}_1\psi_1(\alpha, \beta;q;z) := \sum_{n\in \Z} \frac{(\alpha;q)_n}{(\beta;q)_n}z^n = \frac{(\beta/\alpha;q)_\infty (\alpha z;q)_\infty (q/\alpha z ;q)_\infty (q;q)_\infty}{(q/\alpha ; q)_\infty (\beta/\alpha z ; q)_\infty (\beta;q)_\infty (z;q)_\infty}$$ for $|\beta/\alpha|<|z|<1$ (see, for example, \cite{Gasper}).  Letting $n\mapsto -n$ in $B(\phi;q)$, we see that 
        \[B(\phi;q) = \sum_{n\in\Z} \frac{q^{n^2}(-1; q^2)_n}{q^{-2n}q^{n^2+n}} = \sum_{n\in\Z}  q^n (-1;q^2)_n = {}_1\psi_1(-1, 0; q^2; q).\]
        Similarly, we have
        \[B(\nu;q) = {}_1\psi_1(-q, 0;q^2;q).\]
        Therefore, by Ramanujan's $_1\psi_1$ formula, we can realize these bilateral series as products of $q$-Pochhammer symbols. Then, using the definition of $\eta$, one can show
        \begin{equation*}
            B(\phi;q) = \frac{q^{1/24}\eta(2\tau)^7}{\eta(\tau)^3\eta(4\tau)^3} \hspace{10mm} \text{ and } \hspace{10mm} B(\nu;q) = \frac{2\eta(4\tau)^3}{q^{1/3}\eta(2\tau)^2}.
        \end{equation*}
        Written in this form and given the fact that $B(\psi;q) = \frac{1}{2}B(\phi;q)$, modularity follows with the desired weight and level easily computed.
    \end{proof}

    We now provide radial limits for these mock theta functions using their associated bilateral series.

    \begin{theorem}
        Let $\zeta$ be a primitive root of unity, $k\in\N$, and suppose $q\to\zeta$ radially within the unit disk:
        \begin{itemize}
        	\item[(a)] If $\zeta$ has order $4k$, then we have that
        	\[\lim_{q\to\zeta}\big(\phi(q) - B(\phi;q)\big) = -2\zeta \sum_{n=0}^{k-1} \zeta^n (-\zeta^2; \zeta^2)_n.\]
        	\item[(b)] If $\zeta$ has order $2k-1$, then we have that
        	\[\lim_{q\to\zeta}\big(\psi(q) - B(\psi;q)\big) = -\sum_{n=0}^{k-1} (-1)^n (\zeta; \zeta^2)_n.\]
        	\item[(c)] If $\zeta$ has order $4k-2$, then we have that
        	\[\lim_{q\to\zeta} \big(\nu(q) - B(\nu;q)\big) = -\sum_{n=0}^{k-1} \zeta^n (-\zeta; \zeta^2)_n.\]
        \end{itemize}
       	Moreover, for $M\in\{\phi,\psi,\nu\}$, $B(M;q)$ is a modular form of weight $1/2$ and level $\Gamma_0(4)$.
    \end{theorem}

    \begin{proof}
        Note that by using equation \eqref{Fine61} we have the following relations:
        \begin{align*}
        	2\psi(q) &= 2q\sum_{n\ge 0} q^n (-q^2;q^2)_n, \hspace{10mm}
        	\frac{1}{2}\phi(q) &= \sum_{n\ge0} (-1)^n (q;q^2)_n, \hspace{10mm}
        	\nu(-q) &= \sum_{n\ge0} q^n (-q; q^2)_n.
        \end{align*}
        Now with these equations, the proof of the theorem is similar to that of Theorem \ref{thm:cleanproof}.  The modularity of the bilateral series follows from Lemma \ref{3rdOrderModular}.
    \end{proof}

    The mock theta function $f(q)$ is treated in \cite{FOR1,FOR2}.  In that case, the bilateral series $B(f;q)$ is the product of a modular form and a mock modular form.  In the cases of the 3rd order mock theta functions $\omega$, $\rho$, and $\chi$, as well as for the 7th order mock theta functions, it is not immediately clear how the bilateral series could be utilized. 

\subsection{Even order mock theta functions}

    We now turn our attention to the associated bilateral series for even order mock theta functions.   Using similar methods to Section \ref{modularity}, we can use the linear relations between functions of the same order to prove that the bilateral series are modular forms for eight of the 6th order and six of the 8th order mock theta functions. (See \cite{Gordon} for a full description of these linear relations.) Through arguments similar to those in the proofs of Theorems \ref{thm:cleanproof} and \ref{thm:messyproof}, we then are able to provide the associated closed forms.

    For the sake of brevity, we summarize the relevant information for each of these cases in the tables below. We list only the essential information in order to recreate the corresponding statement associating the closed form to the radial limits.  For example, given the mock theta function $M(q)$, we list the associated bilateral series, the order of the roots of unity for which the radial limit will be taken, and the closed formula for this radial limit.  Given this information we can form the following statement:
    \begin{theorem}
         Let $\zeta$ be an appropriate root of unity with $k\in\N$, and suppose $q \rightarrow \zeta$ radially within the unit disk. We have that
        \[\lim_{q\to\zeta} \big(M(q) - B(M; q)\big) = C(M; \zeta),\]
        where $M(q)$, $C(M; \zeta)$, and $\zeta$ are as in Tables 5.1 and 5.2.  Moreover, $B(M;q)$ is a modular form of weight $1/2$.
    \end{theorem}

\noindent \emph{Remark.} The level and character of $B(M;q)$ can be explicitly calculated using methods as in the proofs of Theorems \ref{thm:cleanproof} and \ref{thm:messyproof}.

Note that Table 5.1 is specifically for mock theta functions of order 6, despite the reuse of notation with mock theta functions of other orders.  For full definitions of these $q$-series, see Section 5 of \cite{Gordon}.  Further, as noted in the remarks after Theorems \ref{thm:messyproof}, the denominators in the closed formulas $C(M;\zeta)$ in the tables never vanish under the hypotheses on the given root of unity $\zeta$.

By writing the bilateral series as linear combinations of mock theta functions, we can then use the linear relations in \cite{Gordon} to show that these series are in fact modular.  For example, for the 6th order mock theta function $\lambda(q)$, we have that $B(\lambda;q) = \lambda(q) + 2\rho(q)$.  Then, from relations (5.8) in \cite{Gordon}, we see that
\[B(\lambda;q) = (q;q^2)_\infty^2 (q;q^6)_\infty (q^5;q^6)_\infty (q^6;q^6)_\infty + 2(-q;q^2)_\infty^2 (-q;q^6)_\infty (-q^5;q^6)_\infty (q^6;q^6)_\infty.\]
Then, using methods similar to those in Section \ref{modularity} (or more generally in \cite{Ono} for example), we can easily recognize this sum of infinite products as a modular form of weight $1/2$.  The proof of all remaining $B(M;q)$ in the tables is similar.

    \begin{center}
    TABLE 5.1: Mock theta functions of order 6\vspace{3mm}
    
        \begin{tabular}{|c|c|c|c|}
        	\hline & & & \\ [-2ex]
        	$M(q)$ & $B(M;q)$ & Order of $\zeta$ & $C(M; \zeta)$ \\ [.5ex] \hline & & & \\ [-1.5ex]
        	    $\lambda$ & $\displaystyle \lambda(q) + 2\rho(q)$ & $2k$ & $\displaystyle -2 \sum_{n=0}^{k-1} \frac{\zeta^{\frac{1}{2} n(n+1)} (-\zeta;\zeta)_n}{(\zeta;\zeta^2)_{n+1}}$\\ [0.2in] \hline & & & \\ [-1.5ex]
            	$\mu$ & $\displaystyle \mu(q) + 2\sigma(q)$ & $2k$ & $\displaystyle -2 \sum_{n=0}^{k-1} \frac{\zeta^{\frac{1}{2}(n+1)(n+2)}(-\zeta;\zeta)_n}{(\zeta;\zeta^2)_{n+1}}$\\ [0.2in] \hline & & & \\ [-1.5ex]$\phi$ & $\displaystyle \phi(q) + 2\nu(q)$ & $2k$ & $\displaystyle -2\sum_{n=0}^{k-1}
            \frac{\zeta^{n+1}(-\zeta;\zeta)_{2n+1}}{(\zeta;\zeta^2)_{n+1}}$\\ [0.2in] \hline & & & \\ [-1.5ex]
            	$\psi$ & $\displaystyle \psi(q) + 2\xi(q)$ & $2k$ & $\displaystyle -2 \sum_{n=0}^{k-1} \frac{\zeta^{n+1} (-\zeta;\zeta)_{2n}}{(\zeta;\zeta^2)_{n+1}}$\\ [0.2in] \hline & & & \\ [-1.5ex]
                $\rho$ & $\displaystyle \rho(q) + \frac{1}{2}\lambda(q)$ & $2k-1$ & $\displaystyle -\frac{1}{2} \sum_{n=0}^{k-1} \frac{(-1)^n \zeta^n (\zeta;\zeta^2)_n}{(-\zeta;\zeta)_n}$\\ [0.2in] \hline & & & \\ [-1.5ex]
            	$\sigma$ & $\displaystyle \sigma(q) + \frac{1}{2}\mu(q)$ & $2k-1$ & $\displaystyle -\frac{1}{2} \sum_{n=0}^{k-1} \frac{(-1)^n (\zeta;\zeta^2)_n}{(-\zeta;\zeta)_n}$\\ [0.2in] \hline & & & \\ [-1.5ex]
            $\nu$ & $\displaystyle \nu(q) + \frac{1}{2}\phi(q)$ & $2k-1$ & $\displaystyle -\frac{1}{2} \sum_{n=0}^{k-1} \frac{(-1)^n \zeta^{n^2} (\zeta;\zeta^2)_n}{(-\zeta;\zeta)_{2n}}$\\ [0.2in] \hline & & & \\ [-1.5ex]
            $\xi$ & $\displaystyle \xi(q) + \frac{1}{2}\psi(q)$ & $2k-1$ & $\displaystyle -\frac{1}{2} \sum_{n=0}^{k-1} \frac{(-1)^n \zeta^{(n+1)^2} (\zeta;\zeta^2)_n}{(-\zeta;\zeta)_{2n+1}}$\\ [0.2in] \hline
        \end{tabular}
    \end{center}

	\newpage
    \begin{center}
    TABLE 5.2: Mock theta functions of order 8\vspace{3mm}
        \begin{tabular}{|c|c|c|c|}
    	\hline & & & \\ [-2ex]
    	$M(q)$ & $B(M;q)$ & Order of $\zeta$ & $C(M; \zeta)$ \\ [.5ex] \hline & & & \\ [-1.5ex]
    	$S_0$ & $\displaystyle S_0(q) + 2T_0(q)$ & $4k$ & $\displaystyle -2 \sum_{n=0}^{k-1} \frac{\zeta^{(n+1)(n+2)} (-\zeta^{2};\zeta^{2})_{n}}{(-\zeta;\zeta^{2})_{n+1}}$\\ [0.2in] \hline & & & \\ [-1.5ex]
    	$S_1$ & $\displaystyle S_1(q) + 2T_1(q)$ & $4k$ & $\displaystyle -2 \sum_{n=0}^{k-1} \frac{\zeta^{n(n+1)} (-\zeta^{2};\zeta^{2})_{n}}{(-\zeta;\zeta^{2})_{n+1}}$\\ [0.2in] \hline & & & \\ [-1.5ex]
    	$T_0$ & $\displaystyle T_0(q) + \frac{1}{2}S_0(q)$ & $4k-2$ & $\displaystyle -\frac{1}{2} \sum_{n=0}^{k-1} \frac{\zeta^{n^{2}} (-\zeta;\zeta^{2})_{n}}{(-\zeta^{2};\zeta^{2})_{n}}$\\ [0.2in] \hline & & & \\ [-1.5ex]
    	$T_1$ & $\displaystyle T_1(q) + \frac{1}{2}S_1(q)$ & $4k-2$ & $\displaystyle -\frac{1}{2} \sum_{n=0}^{k-1} \frac{\zeta^{n(n+2)} (-\zeta;\zeta^{2})_{n}}{(-\zeta^{2};\zeta^{2})_{n}}$\\ [0.2in] \hline  & & & \\ [-1.5ex]
    	$V_0$ & $\displaystyle V_0(q) + V_0(-q) + 1$ & $2k-1$ & $\displaystyle -2 \sum_{n=0}^{k-1} \frac{-\zeta^{n^{2}} (\zeta;\zeta^{2})_{n}}{(-\zeta;\zeta^{2})_{n}}$\\ [0.2in] \hline & & & \\ [-1.5ex]
    	$V_1$ & $\displaystyle V_1(q) - V_1(-q)$ & $2k-1$ & $\displaystyle  \sum_{n=0}^{k-1} \frac{-\zeta^{(n+1)^{2}} (\zeta;\zeta^{2})_{n}}{(-\zeta;\zeta^{2})_{n+1}}$\\ [0.2in] \hline
        \end{tabular}
    \end{center}
\vspace{5mm}

This appears to be an exhaustive list of the mock theta functions for which this method of using bilateral series can be applied to obtain closed formulas for the radial limits.  For those mock theta functions listed in \cite{Gordon} but not included here, the method seems to fail because the bilateral series $B(M;q)$ associated to each of these mock theta functions do not appear to be modular forms. In a small number of cases, we can apply the method of Section \ref{chi0chi1} to obtain a modular form.  For example, for the 8th order mock theta function $U_0$, we can define $\mathcal{B}(U_0;q)$ using the equation $U_0 = S_0(q^2) + qS_1(q^2)$.  However, despite the modularity of these bilateral series, they do not immediately reveal closed formulas for the radial limits.  Therefore, the authors believe an alternate method is needed for the mock theta functions not addressed here.

\bibliographystyle{amsplain}

\begin{thebibliography}{10}

\bibitem{AG}
G. E. Andrews and F. Garvan,
\textit{Ramanujan's ``lost" notebook. VI},
The mock theta conjectures, Adv. in Math. \textbf{73} (1989), no. 2, 242--255.

\bibitem{Fine}
N. J., Fine,
\textit{Basic hypergeometric series and applications},
Math. Surveys and Monographs, Amer. Math. Soc., Providence. (1988), no. 27,

\bibitem{FOR1}
A. Folsom, K. Ono, and R. Rhoades,
\textit{$q$-series and quantum modular forms},
submitted for publication.

\bibitem{FOR2}
A. Folsom, K. Ono, and R. Rhoades,
\textit{Ramanujan's radial limits},
submitted for publication.

\bibitem{Gasper}
G. Gasper and M. Rahman, \emph{Basic hypergeometric series}, Ency. Math. and App. \textbf{35}, Cambridge Univ.
Press, Cambridge, 1990.

\bibitem{Gordon}
B. Gordon and R. McIntosh,
\textit{A survey of the classical mock theta functions, Partitions, $q$-series, and modular forms},
Dev. Math. \textbf{23}, Springer-Verlag, New York, 2012, 95--244.

\bibitem{GOR}
M. Griffin, K. Ono, and L. Rolen
\emph{Ramanujan's mock theta functions}, Proceedings of the National Academy of Sciences, USA, \textbf{110}, No. 15 (2013), pages 5765-5768.

\bibitem{Hickerson}
D. Hickerson,
\textit{A proof of the mock theta conjectures},
Invent. Math. \textbf{94} (1988), no. 3, 639--660.

\bibitem{Kubert}
D. Kubert and  S Lang,
\textit{Modular units},
Springer-Verlag-\textbf{244}, 1981.

\bibitem{Ono}
K. Ono, \emph{The web of modularity: arithmetic of the coefficients of modular forms and q-series}, CBMS
Regional Conference Series in Mathematics, \textbf{102}, Amer. Math. Soc., Providence, RI, 2004.

\bibitem{Watson}
G. N. Watson,
\textit{The mock theta functions},
Proc. London Math. Soc. \textbf{2} (2) (1937), 274--304.

\bibitem{Zagier}
D. Zagier,
\emph{Ramanujan's mock theta functions and their applications (after Zwegers and Ono-Bringmann),} 
S\'eminaire Bourbaki Vol. 2007/2008, 
Ast\'erisque {\textbf{326}} (2009), Exp. No. 986, vii-viii, 143-164 (2010). 

\bibitem{Zwegers01}
S. Zwegers,
\textit{Mock $\vartheta$-functions and real analytic modular forms, q-series with applications to combinatorics, number theory, and physics},
(Ed. B. C. Berndt and K. Ono), Contemp. Math. \textbf{291}, Amer. Math. Soc., (2001), 269-277.

\bibitem{Zwegers02}
S. Zwegers,
\textit{Mock theta functions},
Ph.D. Thesis (Advisor: D. Zagier), Universiteit Utrecht, (2002).



\end{thebibliography}

\end{document}